\documentclass{article}
\usepackage{amsfonts,amsmath,amsthm,amssymb}
\usepackage{graphics, epsfig}
\usepackage{color}
\usepackage{appendix}
\usepackage{ulem}
\usepackage[makeroom]{cancel}
\usepackage{fancyhdr}
\usepackage{centernot}
\usepackage{mathtools}
\usepackage{ stmaryrd }

 \usepackage[usenames,dvipsnames]{pstricks}
 \usepackage{pst-grad} 
 \usepackage{pst-plot} 
\allowdisplaybreaks

\let\TeXchi\chi
\newbox\chibox
\setbox0 \hbox{\mathsurround0pt $\TeXchi$}
\setbox\chibox \hbox{\raise\dp0 \box 0 }
\def\chi{\copy\chibox}

\newtheorem{proposition}{Proposition}[section]
\newtheorem{theorem}{Theorem}[section]

\newtheorem{remark}{Remark}[section]

\numberwithin{equation}{section}
\numberwithin{theorem}{section}
\numberwithin{definition}{section}
\numberwithin{example}{section}
\numberwithin{proposition}{section}
\numberwithin{lemma}{section}
\numberwithin{remark}{section}
\DeclareMathOperator{\Ima}{Im}

\DeclareMathOperator{\id}{id}
\DeclareMathOperator{\Hom}{Hom}
\newcommand\blfootnote[1]{%
  \begingroup
  \renewcommand\thefootnote{}\footnote{#1}%
  \addtocounter{footnote}{-1}%
  \endgroup
}
\begin{document}
\title{A (co)homology theory for some preordered topological spaces}
\author
{Manuel Norman}
\date{}
\maketitle
\begin{abstract}
\noindent The aim of this short note is to develop a (co)homology theory for topological spaces together with the specialisation preorder. A known way to construct such a (co)homology is to define a partial order on the topological space starting from the preorder, and then to consider some (co)homology for the obtained poset; however, we will prove that every topological space with the above preorder consists of two disjoint parts (one called 'poset part', and the other one called 'complementary part', which is not a poset in general): this suggests an improvement of the previous method that also takes into account the poset part, and this is indeed what we will study here.
\end{abstract}
\blfootnote{Author: \textbf{Manuel Norman}; email: manuel.norman02@gmail.com\\
\textbf{AMS Subject Classification (2010)}: 06A11, 55U10\\
\textbf{Key Words}: preorder, poset, (co)homology}
\section{Introduction}
We start describing the motivation for which we develop a (co)homology theory for topological spaces together with the specialisation preorder, i.e. the preorder defined by:
\begin{equation}\label{Eq:1.1}
x \leq y \Leftrightarrow x \in \overline{\lbrace y \rbrace}
\end{equation}
where $\overline{\lbrace y \rbrace}$ denotes the closure of the singleton set containing $y$, that is, the intersection of all the closed subsets containing such singleton (throughout this paper, $\leq$ will always indicate the specialisation preorder). We will show (see Section 2) that every topological space together with the above preorder can be assigned a poset $X / \sim$, which can be seen as an abstract subset of $X$, where the partial order is given by $\leq$ defined above. This will be called the 'poset part'. The 'complementary part' is not, in general, a poset. A usual way to obtain a (co)homology theory for a topological space under the specialisation preorder is to consider $X$ with the following partial order $\preceq$:
$$ x \prec y \Leftrightarrow x \leq y \, \text{and not} \, y \leq x $$
$$ x \preceq y \Leftrightarrow x \prec y \, \text{or} \, x=y $$
Then, $(X, \preceq)$ is a poset, and thus we can associate it various possible (co)homologies: we can assign it the order complex $\Delta(X)$ and evaluate its simplicial (co)homology, we can (if the poset satisfies some conditions) consider a cohomology with coefficients in some presheaf, we can take some colouring and consider its coloured poset (co)homology (if some assumptions are satisfied), ... (see, for instance, [1-6] and [9]). However, this constructs a new partial order starting from $\leq$; we would like to have a (co)homology that also takes into account the preorder $\leq$ itself. \textit{We will hence develop a (co)homology theory that also involves this} \textbf{preorder}, \textit{which turns out to be a partial order} \textbf{on the poset part} \textit{of} $X$, \textit{and we will consider the} \textbf{usual method} \textit{only} \textbf{on the complementary part}. Therefore, we will obtain an "improvement" of the usual (co)homology by "adding together" (in some way) the two parts of $X$.
\section{(Co)homology theory for $(X, \leq)$}
We begin this section proving some results that will be needed to develop our (co)homology theory. First of all, we will construct the poset and the complementary parts of $(X, \leq)$ (henceforth, $X$ will always indicate a topological space with the specialisation preorder $\leq$ or with the partial order $\preceq$). The idea of the construction of $\sim$ is similar to the one used in Section 4 of [12]: we delete the elements that violate antisymmetry \footnote{Actually, this construction for preordered spaces is well known. The new idea contained in this note is to use a 'structured approach' (see also Remark \ref{Rm:2.1}) with the "decomposition" of a preordered space into its poset and complementary parts.}.
\begin{proposition}\label{Prop:2.1}
Let $(X, \leq)$ be a topological space with the specialisation preorder. Define the equivalence relation $\sim$ on $X$ by (the closure is always w.r.t. $X$):
\begin{equation}\label{Eq:2.1}
x \sim y \Leftrightarrow x \leq y \, \text{and} \, y \leq x
\end{equation}
Then, $X/ \sim$ is a poset under $\leq$, and it can be seen as an abstract subset of $X$. This is called the 'poset part' of $X$. Its complementary w.r.t. $X$, i.e. $X \setminus (X/ \sim)$, is the 'complementary part' of $X$, and it is not, in general, a poset under $\leq$.
\end{proposition}
\begin{proof}
First of all, notice that $X/ \sim$ can be seen as an abstract subset of $X$ by associating to each equivalence class $[a] \in X/ \sim$ some element $b \in X$ such that $[a]=[b]$. It is then clear that $\leq$ is reflexive and transitive on $X/ \sim$, because these properties hold on $X$. Moreover, whenever $x \leq y$ and $y \leq x$, we have $x \sim y$ (by definition), so that $X/ \sim$ is also antisymmetric and thus a poset. The fact that the complementary part is not, in general, a poset, easily follows since if at least one element $x$ of $X$ is such that there exist two (or more) elements $y$ and $z$ with $[x]=[y]=[z]$ and $x \neq y$, $y \neq z$, $x \neq z$, then clearly only one of these elements belong to $X / \sim$, while the other two are in $X \setminus (X / \sim)$. Consequently, antisymmetry cannot hold on the complementary part in such cases.
\end{proof}
There are however some cases where the complementary part is a poset; the (co)homology theory that we will develop continues to work properly even in such situations. Henceforth, we will not consider the spaces $(X, \leq)$ which turn out to actually be posets. This is because, in such cases, the equivalence relation $\sim$ is actually $=$, so we can just use the (co)homologies for the poset, instead of the one shown in this paper.\\
Before introducing our (co)homology, we need one more result. First of all, recall that the order complex $\Delta(P)$ associated to a poset $P$ is the simplicial complex whose vertices are the points in $P$ and whose faces are the totally ordered subsets (called 'chains') of $P$. A subcomplex of a simplicial complex is a subset which is itself a simplicial complex. We now prove the following result:
\begin{proposition}\label{Prop:2.2}
Let $X$ be a topological space together with the specialisation preorder. Define the partial order $\preceq$ on $X$ as in Section 1. Then, the order complex $\Delta(X/ \sim) |_{\preceq}$ of the poset $X/ \sim$ (the notation stresses the fact that, this time, $X/ \sim$ is considered under $\preceq$) is a subcomplex of the order complex $\Delta(X)$ of the poset $X$ (under $\preceq$).
\end{proposition}
\begin{proof}
First of all, we verify that $X/ \sim$ is a poset also under $\preceq$. By Proposition \ref{Prop:2.1}, we know that $X/ \sim \, \, \subseteq X$. But nonempty subsets of posets are posets under the (induced) relation, in this case $\preceq$ (this can be easily verified), and thus $X/ \sim$ is a poset under both $\leq$ and $\preceq$ (here we need to consider the latter one).
Now, since $X/ \sim \, \, \subseteq X$, by their definition and by this inclusion, $\Delta(X / \sim) \subseteq \Delta(X)$. Indeed, since we are now considering $X/ \sim$ under the induced relation $\preceq$, it is clear that the chains (that is, the totally ordered subsets, under $\preceq$) of $X/ \sim$ must be also contained in $\Delta(X)$. Thus, the above inclusion follows, and the fact that these two sets are both simplicial complexes completes the proof of this Proposition.
\end{proof}
We can therefore consider the relative simplicial (co)homology for these two simplicial complexes (see, for instance, [7-8]). The homology groups $H_p(X, X/ \sim)$ "delete" $X/ \sim$ and thus "represent" the complementary part. By applying, as usual, the functor $\Hom( \cdot, G)$ for some abelian group $G$, we obtain a cohomology from this homology. The idea is to "sum up" the actions of the poset part and of the complementary part as follows (let $G^n _1$ be the groups in the cochain of $\Delta(X/ \sim)|_{\leq}$ and $G^n _2$ be the groups in the cochain of the relative cohomology; homology can be studied similarly):
$$ 0 \rightarrow G^0 _1 \xrightarrow{\partial^0 _1} G^1 _1 \xrightarrow{\partial^1 _1} G^2 _1 \qquad \qquad \qquad \, \, G^3 _1 \xrightarrow{\partial^3 _1} G^4 _1 \xrightarrow{\partial^4 _1} ...$$
$$ \qquad \qquad  \downarrow 0 \qquad \qquad \quad \, \,  \uparrow 0 $$
$$ \qquad \quad \, G^0 _2 \xrightarrow{\partial^0 _2} G^1 _2 \xrightarrow{\partial^1 _2} G^2 _2  $$
where $0$ denotes the trivial homomorphism \footnote{We will see, in Theorem \ref{Thm:2.1}, that using the trivial homomorphisms does not affect negatively the (co)homology obtained: indeed, it still gives us the information we wanted and actually it neither gives rise (generally) to trivial groups $(0)$. A similar construction has been also performed in [13]; see there for more details on the use of trivial homomorphisms (the situation can be easily adapted to the above case).}, and the other homomorphisms are the usual ones given by the corresponding cochain. It is clear that, this way, we obtain a new cochain which consitutes a cohomology theory for $X$, and which takes into account $X/ \sim$ under $\leq$ (this is the first cochain) and also the complementary part, represented by the relative cohomology, under $\preceq$ (the second cochain). This gives rise to the theory we were looking for. Actually, we can even be more precise on the (co)homology groups obtained. We will prove the following final result for cohomologies; the result for homologies is analogous.
\begin{theorem}\label{Thm:2.1}
Let $X$ be a topological space together with the specialisation preorder $\leq$. Consider the poset $(X/ \sim)|_{\leq}$ (under $\leq$) and the poset $X$ under $\preceq$. Then, we can define the cochain complex:
\begin{equation}\label{Eq:2.2}
0 \rightarrow G^0 _1 \xrightarrow{\partial^0 _1} G^1 _1 \xrightarrow{\partial^1 _1} G^2 _1 \xrightarrow{0}  G^0 _2 \xrightarrow{\partial^0 _2} G^1 _2 \xrightarrow{\partial^1 _2} G^2 _2 \xrightarrow{0} G^3 _1 \xrightarrow{\partial^3 _1} G^4 _1 \xrightarrow{\partial^4 _1} ...
\end{equation}
where we use both the cohomology of $\Delta(X/ \sim)|_{\leq}$ and the relative simplicial cohomology of $\Delta(X)$ (under $\preceq$) and $\Delta(X/ \sim)|_{\preceq}$ (which represents the complementary part). The cohomology groups obtained are (we indicate them by $\widehat{H}^p(X)$), for $p \in \mathbb{N}_0$:\\
$$ \widehat{H}^{6p}(X) = H^{3p}(X/ \sim) $$
$$ \widehat{H}^{6p+1}(X) = G^{3p+2} _1 / \Ima \partial^{3p+1} _1 $$
$$ \widehat{H}^{6p+2}(X) \cong \ker \partial^{3p} _2$$
$$ \widehat{H}^{6p+3}(X) = H^{3p}(X \setminus (X/ \sim)) $$
$$ \widehat{H}^{6p+4}(X) = G^{3p+2} _2 / \Ima \partial^{3p+1} _2 $$ 
$$ \widehat{H}^{6p+5}(X) \cong \ker \partial^{3p+3} _1 $$
(where we used $X/ \sim$ and $X \setminus (X/ \sim)$ to indicate the simplicial cohomology of $\Delta(X/ \sim)|_{\leq}$ and the relative simplicial cohomology of $\Delta(X)$ and $\Delta(X/ \sim)|_{\preceq}$, respectively).
\end{theorem}
\begin{proof}
The cohomology groups are simply evaluated by their definition; in some cases it is also used the well known fact that:
$$ A/ \lbrace \id \rbrace \cong A$$
(for some group $A$) which can be proved, for instance, via the first Isomorphism Theorem (considering the identity map).
\end{proof}
An important aspect to notice is that $\Delta(X/ \sim)|_{\leq}$ is defined via $\leq$, while the relative (co)homology involves $\preceq$ (both for $\Delta(X)$ and for $\Delta(X/ \sim)|_{\preceq}$, which is a subcomplex considered, in this case, w.r.t. $\preceq$).\\
We will also use this (co)homology in a paper on structured spaces (see also [12]) which is now in preparation. Actually, we also notice that, in a future paper, we will discuss Eilenberg-Steenrod axioms for square/rectangular (co)homology; the same observations also hold for the above (co)homology \footnote{This implies that this "improved" (co)homology gives interesting information about both the poset and the complementary part, thanks to some important aspects guaranteed by the ES axioms.}. In order to have the ES axioms satisfied locally, in a sense that will be made clear in such paper, it is useful to slightly modify the definition of the (co)homology above so that the "local horizontal (co)chains" (i.e. the ones which are only connected by the horizontal homomorphisms, and not by the vertical ones) are allowed to be longer. This means, for instance, that we can define a cohomology as follows:
$$ 0 \rightarrow G^0 _1 \xrightarrow{\partial^0 _1} G^1 _1 \xrightarrow{\partial^1 _1} G^2 _1 \xrightarrow{\partial^2 _1} G^3 _1 \qquad \qquad \qquad \qquad \quad \,  G^4 _1 \xrightarrow{\partial^4 _1} ...$$
$$ \qquad \qquad \qquad \qquad  \quad \, \downarrow 0 \qquad \qquad \qquad \qquad \,  \uparrow 0 $$
$$ \qquad \qquad \qquad \qquad  G^0 _2 \xrightarrow{\partial^0 _2} G^1 _2 \xrightarrow{\partial^1 _2} G^2 _2 \xrightarrow{\partial^2 _2} G^3 _2  $$
and so on. The cohomology of Theorem \ref{Thm:2.1} is said to be 'of length $3$', while the above one is said to be 'of length $4$'. It is then easy to construct various (co)homologies of length $n$, so that ES are locally satisfied. We also note that the limit of the length, for $n \rightarrow +\infty$, can be seen as the usual (co)homology of $\Delta(X / \sim)$, while the limit for $n \rightarrow - \infty$ (where we adopt the convention that the negative length changes the order of the groups above, that is, the first line contains the groups of the relative (co)homology, and the second one the groups of $\Delta(X/ \sim)$) is the usual relative (co)homology.
\begin{remark}\label{Rm:2.1}
\normalfont The idea of 'summing up' some (co)homology theories comes from the context of structured spaces and from square/rectangular (co)homologies. It is not surprising that this idea can be similarly applied to "mixed structures", where we generally consider a construction as the above one for each "component" of the given mixed object. For instance, if we have a partially ordered group $T$, we can assign it some group cohomology (choose some $T$-module), and (supposing it is also graded) some different cohomology theories for posets. More precisely, if $G^n _i$, $i=1,...,4$ represent the groups in the cochain of group cohomology, simplicial cohomology of the order complex of the poset, cellular cohomology for the poset (assumed to be graded) and coloured cohomology, respectively, then we can sum up everything as follows (again, we can change the length $n$, which below is equal to $3$, so that the ES axioms are locally satisfied):
$$ 0 \rightarrow G^0 _1 \xrightarrow{\partial^0 _1} G^1 _1 \xrightarrow{\partial^1 _1} G^2 _1  \qquad \qquad \qquad \qquad \qquad \qquad \qquad \qquad \qquad \qquad \qquad \qquad \quad \, $$
$$  \, \downarrow 0 \qquad \qquad \qquad \qquad \qquad \qquad \qquad \quad $$
$$   G^0 _2 \xrightarrow{\partial^0 _2} G^1 _2 \xrightarrow{\partial^1 _2} G^2 _2 \qquad \qquad \qquad \qquad$$
$$ \qquad \qquad \qquad \qquad  \quad \, \downarrow 0 \qquad \qquad \qquad \qquad \qquad \,  $$
$$ \qquad \qquad \qquad \qquad \qquad \qquad \qquad \qquad  G^0 _3 \xrightarrow{\partial^0 _3} G^1 _3 \xrightarrow{\partial^1 _3} G^2 _3 \qquad \qquad \qquad \quad G^3 _3 \xrightarrow{\partial^3 _3} ...$$
$$ \qquad \qquad \qquad \qquad \qquad \qquad \qquad \qquad \qquad \qquad \quad  \downarrow 0 \qquad \qquad \qquad \uparrow 0   $$
$$ \qquad \qquad \qquad \qquad \qquad \qquad \qquad \qquad \qquad \qquad \quad G^0 _4 \xrightarrow{\partial^0 _4} G^1 _4 \xrightarrow{\partial^1 _4} G^2 _4   $$
Something similar can be done in the general case. This 'structured approach' can be often useful when dealing with objects that are "mixed" in some sense.
\end{remark}
\section{Conclusion}
In this short note we have provided an "improvement" of a (co)homology theory for topological spaces $X$ preordered via the specialisation preorder. This is done considering the poset part of such a topological space, i.e. an abstract subset which turns out to be always a poset under the specialisation order. Then, we "sum up" the actions of the this part and the complementary part of $X$ and find an interesting (co)homology theory. This idea, which arises from the context of structured spaces, can also be applied in more general situations, as explained in Remark \ref{Rm:2.1}.\\
\\
\begin{large}
\textbf{References}
\end{large}
\\
$[1]$ Wachs, M. L. (2007). Poset topology: tools and applications. Geometric Combinatorics, IAS/Park City Math. Series 13, (Miller, Reiner, and Sturmfels, eds.), American Mathematical Society, Providence, RI,497-615\\
$[2]$ Everitt, B.; Turner, P. (2009). Homology of coloured posets: A generalisation of Khovanov's cube construction. Journal of Algebra, Vol. 322, No. 2, 429-448\\
$[3]$ Everitt, B.; Turner, P. (2015). Cellular cohomology of posets with local coefficients. Journal of Algebra, Vol. 439, 134-158\\
$[4]$ Cianci, N.; Ottina, M. (2019). Homology of posets with functor coefficients and its relation to Khovanov homology of knots. Preprint (arXiv:1907.03974)\\
$[5]$ Brun, M.; Bruns, W.; R\"omer, T. (2007). Cohomology of partially ordered sets and local cohomology of section rings. Advances in Mathematics, 208(1): 210-235\\
$[6]$ D\'iaz, A. (2007). A method for integral cohomology of posets. Preprint (arXiv:0706.2118)\\
$[7]$ Gallier, J.; Quaintance, J. (2019). A Gentle Introduction to Homology, Cohomology, and
Sheaf Cohomology, University of Pennsylvania\\
$[8]$ Robinson, M.; Capraro, C.; Joslyn, C.; Purvine, E.; Praggastis, B.; Ranshous, S.; Sathanur, A. V. (2018). Local homology of abstract simplicial complexes. Preprint (arXiv:1805.11547)\\
$[9]$ Barmak, J. A.; Minian, E. G. (2012). $G$-colorings of posets, coverings and presentations of the fundamental group. Preprint (arXiv:1212.6442)\\
$[10]$ Bradley, P. E.; Paul, N. (2013). Dimension of Alexandrov Topologies. Preprint (arXiv:1305.1815)\\
$[11]$ Arenas, F. (1999). Alexandroff spaces. Acta Math. Univ. Comenianae Vol. LXVIII, 1 , 17-25\\
$[12]$ Norman, M. (2020). On structured spaces and their properties. Preprint (arXiv:2003.09240)\\
$[13]$ Norman, M. (2020). Two cohomology theories for structured spaces. In preparation\\
$[14]$ Raptis, G. (2010). Homotopy theory of posets. Homology Homotopy Appl. 12, no. 2, 211-230\\
$[15]$ Brun, M.; Bruns, W.; R\"omer, T. (2007). Cohomology of partially ordered sets and local cohomology of section rings. Advances in Mathematics, 208(1): 210-235\\
$[16]$ Curry, J. (2019). Functors on Posets Left Kan Extend to Cosheaves: an Erratum. Preprint (arXiv:1907.09416)\\
$[17]$ Schr\"oder, B. (2016). Ordered sets: An Introduction with Connections from Combinatorics to Topology, 2nd edition. Birkh\"auser\\
$[18]$ Folkman, J. (1966). The homology groups of a lattice. J. Math. Mech. 15, 631-636\\
$[19]$ Brown, K. S. (1972). Cohomology of Groups. Graduate Texts in Mathematics, 87, Springer-Verlag

\end{document}